\documentclass[b5paper
] {article}
[12pt]

\usepackage{CJK}
\usepackage{amsmath}
 
\usepackage{amsfonts}
\usepackage{amssymb}
\usepackage{amsthm}
\usepackage{amssymb}
\usepackage{enumerate}
\usepackage[calc]{picture}
\usepackage[all,cmtip]{xy}

\usepackage[mathscr]{eucal}
\usepackage{eqlist}

\usepackage{color}
\usepackage{abstract} 
\usepackage[T1]{fontenc}

\usepackage{tikz}
\usepackage{dsfont}
\usetikzlibrary{matrix}

\usepackage[top=2.18cm, bottom=2.18cm, left=1.58cm, right=1.58cm]{geometry}

\theoremstyle{plain}
\newtheorem{theorem}{Theorem}
\newtheorem{proposition}[theorem]{Proposition}
\newtheorem{lemma}[theorem]{Lemma}
\newtheorem{corollary}[theorem]{Corollary}

\theoremstyle{definition}

\theoremstyle{remark}
\newtheorem{remark}[theorem]{Remark}

\numberwithin{equation}{section}
\numberwithin{theorem}{section}

\begin{document}
\title{Order of the canonical vector bundle  over configuration spaces  of disjoint unions of spheres}
 \author{Shiquan Ren}

\begin{center}
{\Large {\textbf{{ Order of the Canonical Vector Bundle  over Configuration Spaces  of Spheres 
}}}}
 \vspace{0.3cm}

Shiquan Ren 

\vspace{0.5cm}

\end{center}

{\small
\begin{quote}
\begin{abstract}
\medskip

Given a vector bundle, its (stable)  order is the smallest positive integer $t$ such that the $t$-fold self-Whitney sum is (stably) trivial.  
So far, the order and the stable order of the canonical vector bundle over configuration spaces of Euclidean spaces have been studied by F.R. Cohen, R.L. Cohen, N.J. Kuhn and J.L. Neisendorfer \cite{bundle1983}, F.R. Cohen, M.E. Mahowald and R.J. Milgram \cite{bundle1978}, and S.W. Yang \cite{swy}. Moreover, the order and the stable order of the canonical vector bundle over configuration spaces of closed orientable Riemann surfaces with genus greater than or equal to one have been studied by F.R. Cohen, R.L. Cohen, B. Mann and  R.J. Milgram \cite{bundle1989}. 
In this paper, we mainly study the order and the stable order of the canonical vector bundle over configuration spaces of spheres and disjoint unions of spheres.   
 \end{abstract}
\end{quote}
}
 \bigskip
 
\noindent{\bf{AMS Mathematical Classifications 2010}}.  	 Primary  55R80, 55R10;     Secondary    55P15,  55P40.

\smallskip


   \smallskip

\noindent{\bf{Acknowledgement}}. The present author would like to express his deep gratitude to Professor Frederick R. Cohen and Professor Jie Wu 
for their kind guidance  and helpful encouragement  on this topic. The present author would like to thank the referee for the valuable comments and helpful suggestions.

\medskip

\section{Introduction}

For a vector bundle $\xi$ and a positive integer $t$, we denote the $t$-fold Whitney sum of itself  by $\xi^{\oplus t}$. If there exists a positive integer $t$ such that $\xi ^{\oplus t}$ is  trivial, then $\xi$ is said to have finite order and the smallest such $t$  is called the order of $\xi $, denoted by $o(\xi)$.

Let $\epsilon^t$ denote a trivial vector bundle of dimension $t$. A vector bundle $\xi$ is called stably trivial if there exist positive integers  $r$ and $s$ such that the Whitney sum $\xi\oplus \epsilon^r$ is isomorphic to $\epsilon^{s}$. If there exists a positive integer $t$ such that $\xi ^{\oplus t}$ is  stably trivial, then $\xi$ is said to have finite stable order and the smallest such $t$  is called the stable order of $\xi $, denoted by $s(\xi)$.  It follows from the definition that the stable order of a vector bundle must divide the order. 

\smallskip

Let $(M,M_0)$ be a relative $CW$-complex and $k$  be  a positive integer.  The configuration space $F(M|M_0,k)$ is the subspace of the Cartesian product $M^k$ consisting of points represented by $(x_1,x_2,\cdots,x_k)$ such that $x_i\neq x_j$ if $i\neq j$ and $x_i\in M_0$ for some $1\leq i\leq k$. The symmetric group on $k$-letters, denoted by $\Sigma_k$, acts on $F(M|M_0,k)$ from the left by 
\begin{eqnarray*}
\sigma(x_1,x_2,\cdots,x_k)=(x_{\sigma(1)},x_{\sigma(2)},\cdots,x_{\sigma(k)}) ,\text{\ \ \  } \sigma\in \Sigma_k.
\end{eqnarray*}
 This action is free and induces a covering map from $F(M|M_0,k)$ to $F(M|M_0,k)/\Sigma_k$.  
The associated vector bundle of this covering map is 
\begin{eqnarray}\label{01292016-e1}
\xi_{M|M_0,k}: \mathbb{R}^k\longrightarrow F(M|M_0,k)\times_{\Sigma_k}\mathbb{R}^k\longrightarrow F(M|M_0,k)/\Sigma_k   
\end{eqnarray}
where $\Sigma_k$ acts  on $\mathbb{R}^k$ by permuting the coordinates from the right. In particular, if $M_0=M$, then the configuration space  $F(M|M_0,k)$ and  the bundle   (\ref{01292016-e1}) 	are simply denoted as $F(M,k)$  and    $\xi_{M,k}$ respectively.

\smallskip

The order  and the stable order of          $\xi_{M,k}$            
 have been extensively studied when $M$ is a Euclidean space. In 1978,  F.R. Cohen, M.E. Mahowald and R.J. Milgram \cite[Theorem~1.2]{bundle1978}  proved that  the order of $\xi_{\mathbb{R}^2,k}$ is $2$. In 1981, S.W. Yang \cite[Theorem~1.1, Theorem~1.2]{swy} proved that for any odd prime $p$, the stable order of $\xi_{\mathbb{R}^m,p}$ is $p^{[(m-1)/2]}s$ where $s$ is prime to $p$.  Moreover, if $p^n\leq k<p^{n+1}$, then the stable orders of $\xi_{\mathbb{R}^m,p^n}$ and $\xi_{\mathbb{R}^m,k}$ are divisible by the same power of $p$.  For a positive integer $t$, let $\rho(t)$ be the number of positive integers less than or equal to $t$ that are congruent to $0,1,2$ or $4$ mod $8$. In 1983,  F.R. Cohen, R.L. Cohen, N.J. Kuhn and J.L. Neisendorfer \cite[Theorem~1.1]{bundle1983} generalized \cite[Theorem~1.2]{bundle1978} and  \cite[Theorem~1.1, Theorem~1.2]{swy}.  They  proved that for any $k\geq 2$, if $m$ is not divisible by $4$, then the stable order of $\xi_{\mathbb{R}^m,k}$ is 
\begin{eqnarray*}
a_{m,k}=2^{\rho(m-1)}\prod_{3\leq p\leq k, \text{\ \  } p\text {  prime}} p^{[\frac{m-1}{2}]};
\end{eqnarray*}
 and if  $m$ is divisible by $4$, then the stable order of $\xi_{\mathbb{R}^m,k}$ is either $a_{m,k}$ or $2a_{m,k}$.

Besides the Euclidean-space case, the order of  $\xi_{M,k}$  has also been studied when $M$ is a surface.  In 1989,  F.R. Cohen, R.L. Cohen, B. Mann and  R.J. Milgram \cite[Proposition~1.1]{bundle1989} proved  that for any closed orientable Riemann surface $S$ whose genus is greater than or equal to one, both $\xi_{S,k}$ and $\xi_{S\setminus\{\text{point}\},k}$ have order $4$. 

\smallskip

In this paper, generalizing \cite[Theorem~1.1]{bundle1983}, we will compute in Theorem~\ref{period-t1} the order and the stable order of (\ref{01292016-e1}) when $M$ is a Euclidean space. We will prove that for any $m\geq 2$, the order as well as the stable order of $\xi_{\mathbb{R}^m,k}$ is $a_{m,k}$.  Then  supplementary to \cite{bundle1983, bundle1989, bundle1978, swy}, we will study the order as well as the stable order of (\ref{01292016-e1}) when  $M$ is a sphere or a disjoint union of spheres. We will prove  the following theorem.
 
 \begin{theorem}[Main Theorem]\label{main}
Let  $n$ be a positive integer,  $M_0$ be a non-empty $CW$-subcomplex of $\coprod_nS^m$ and $k\geq 2$.  Then we have the following.

(a). The order and the stable order of $\xi_{\coprod_nS^m|M_0,k}$ are equal.

(b). The order of $\xi_{\coprod_n S^m,k}$ is either $a_{m,k}$ or $2^{\rho(m)-\rho(m-1)} a_{m,k}$.  Moreover, if $k$ is nonprime, then the order of $\xi_{\coprod_nS^m|M_0,k}$ is either $a_{m,k}$ or $2^{\rho(m)-\rho(m-1)} a_{m,k}$.

(c). Suppose either (i). $n=1$, $m=2$ and $k$ is even, or (ii). $n\geq  2$. Then the order of $\xi_{\coprod_nS^m,k}$ is $2^{\rho(m)-\rho(m-1)} a_{m,k}$. Moreover, if $k$ is nonprime, then the order of $\xi_{\coprod_nS^m|M_0,k}$ is  $2^{\rho(m)-\rho(m-1)} a_{m,k}$.
\end{theorem}

The next corollary is a particular case of Theorem~\ref{main}~(b).

\begin{corollary}\label{0318-c1}
Let  $n$ be a positive integer and $k\geq 2$.  If $m\equiv 3, 5, 6$ or $7$ mod $8$, then the order of $\xi_{\coprod_n S^m,k}$ is $a_{m,k}$. Moreover, if $k$ is nonprime and $M_0$ is a non-empty $CW$-subcomplex of $\coprod_nS^m$, then the order of $\xi_{\coprod_nS^m|M_0,k}$ is $a_{m,k}$. 
\end{corollary}


The next corollary is a direct consequence of Theorem~\ref{main}~(b).

\begin{corollary}\label{2017-c1}
Let  $n$ be a positive integer and $k\geq 2$.  Let $r$ be an   integer greater than or equal to $\rho(m)$.  Then   the order of $(\xi_{\coprod_n S^m,k})^{\oplus 2^r}$ is $\prod_{3\leq p\leq k, \text{\ \  } p\text {  prime}} p^{[\frac{m-1}{2}]}$.  Moreover, if $k$ is nonprime and $M_0$ is a non-empty $CW$-subcomplex of $\coprod_nS^m$, then the order of $(\xi_{\coprod_n S^m\mid M_0,k})^{\oplus 2^r}$ is $\prod_{3\leq p\leq k, \text{\ \  } p\text {  prime}} p^{[\frac{m-1}{2}]}$. 
\end{corollary}

From Section~\ref{2.0} to Section~\ref{2.2}, we   give some preliminaries and auxiliary lemmas for the preparation of the proof of Theorem~\ref{main}. And in Section~\ref{sss4},  we   prove Theorem~\ref{main}. 


\smallskip

Given a relative $CW$-complex $(M,M_0)$ and a space $X$ with non-degenerate base-point, we have the $k$-adic constructions $D_k(M,M_0;X)$ (cf. \cite[Section 2.2]{cohen1989}), which  will be denoted as $D_k(M;X)$ if $M_0$ is the empty set.  
So far, the stable homotopy types of $D_k(\mathbb{R}^m; S^n)$  have been studied in \cite{bundle1983,bundle1978,1994}. And the stable homotopy types of $D_k(S; S^n)$ and $D_k(S\setminus\{\text{point}\}; S^n)$, where $S$ is a closed orientable Riemann surface with genus  greater than or equal to one,  have been studied in \cite{bundle1989}.

Motivated by  \cite{bundle1983, bundle1989, bundle1978, 1994}, we  give some by-products of Section~\ref{sss4} and apply the order of (\ref{01292016-e1}) to study the stable homotopy type of the $k$-adic constructions.  In order to do this, we review some lemmas and give some auxiliary results in Section~\ref{2.3}.  In Section~\ref{sss5}, we give some periodicity properties of the stable homotopy types of $D_k(M,M_0;\Sigma^nX)$ when $M$ is a Euclidean space, a hypersurface in Euclidean spaces and a disjoint union of spheres  in Proposition~\ref{adic-c1} - Proposition~\ref{0215-c1}  respectively.

\smallskip

Throughout this paper, all maps are assumed to be continuous.  All manifolds, including hypersurfaces in Euclidean spaces, are assumed to be finite $CW$-complexes and have dimensions at least $2$.

\medskip

 
 
 \section{The canonical vector bundle over configuration spaces}\label{2.0}
 
We prove some lemmas on   the canonical vector bundle over configuration spaces as well as its order and stable order.  
 
 \smallskip

 The following lemma proves that if $M$ is a connected manifold, then the order and the stable order of $\xi_{M,k}$ are equal.
  
\begin{lemma}\label{0321-l2}
Let $M$ be a connected $m$-dimensional manifold without boundary, $m\geq 2$, and $M_0$ be a non-empty $CW$-subcomplex of $M$. Then for any $k\geq 2$, the order and the stable order of $\xi_{M|M_0,k}$ are equal.
\end{lemma}

\begin{proof}
Let $f: F(M|M_0,k)/\Sigma_k\longrightarrow BO(k)$ and 
$
g: F(M|M_0,k)/\Sigma_k \longrightarrow BO
$
be the classifying map and the stable classifying map  of $\xi_{M|M_0,k}$ respectively.  Let $s$ be the stable order of $\xi_{M|M_0,k}$. We divide our proof into two steps. 

{\it Step~1}.  The mapping space $[F(M|M_0,k)/\Sigma_k; O/O(ks)]$ is trivial. 

{\it Proof of Step~1}. 
Since $M$ is a connected manifold (with its dimension at least $2$), $F(M,k)/\Sigma_k$ is   connected as well. And since $M$ is without boundary, $F(M,k)/\Sigma_k$ is a $km$-dimensional open manifold.  Thus as a $CW$-complex, $F(M,k)/\Sigma_k$ does not have any cells whose dimensions are greater than or equal to $km$. Moreover,  since  $F(M|M_0,k)/\Sigma_k$ is a $CW$-subcomplex of $F(M,k)/\Sigma_k$,  the dimension of  $F(M|M_0,k)/\Sigma_k$  (as a $CW$-complex) is smaller than or equal to $km-1$.  
On the other hand, it follows from the fibrations
\begin{eqnarray*}
O(t-1)\longrightarrow O(t)\longrightarrow S^{t-1}, \text{\ \ \ } t=1,2,\cdots,ks,
\end{eqnarray*}
  that $O/O(ks)$ does not have any cells of dimensions $1,2,\cdots, ks-1$. 
  Since
$s\geq s(\xi_{\mathbb{R}^m,k})
\geq  m$,
it follows that $[F(M|M_0,k)/\Sigma_k; O/O(ks)]$ is trivial.

{\it Step~2}. $f^{\oplus s}$ is null-homotopic.

{\it Proof of Step~2}. 
Since $s$ is the stable order,  the map
\begin{eqnarray*}
\xymatrix{
&g^{\bigoplus s}: F(M|M_0,k)/\Sigma_k\ar[rr]^{\Delta_s} && \prod_s  F(M|M_0,k)/\Sigma_k\\
&\ar[r]^{\prod_s g }&\prod_s BO\ar[r]^{\mu} &BO &&
}
\end{eqnarray*}
is null-homotopic.  With the help of  the fibration 
\begin{eqnarray*}
O/O(ks)\longrightarrow BO(ks)\longrightarrow BO,
\end{eqnarray*}
we see that the map 
\begin{eqnarray*}
\xymatrix{
&f^{\bigoplus s}: F(M|M_0,k)/\Sigma_k\ar[rr]^{\Delta_s}  && \prod_s  F(M|M_0,k)/\Sigma_k\\&\ar[r]^{\prod_s f }& \prod_s BO(k)\ar[r] &BO(ks) &&
}
\end{eqnarray*}
 can be lifted to  a  map $h$ from $F(M|M_0,k)/\Sigma_k$ to $O/O(ks)$.  It follows from Step~1 that $h$ is null-homotopic. Thus $f^{\oplus s}$ is null-homotopic as well.   
 
Therefore, by Step~2, the order of $\xi_{M|M_0,k}$ equals to $s$ and the assertion follows.  
\end{proof}

\begin{remark}\label{remark2.2}
In general, suppose $Y$ is a $CW$-complex with a free $\Sigma_k$-action, and $Y/\Sigma_k$ is a finite dimensional CW-complex. Then we have a canonical vector bundle
\begin{eqnarray*}
\xi_Y: \mathbb{R}^k\longrightarrow Y\times _{\Sigma_k}\mathbb{R}^k\longrightarrow Y/\Sigma_k. 
\end{eqnarray*}
By an analogous argument of Step~1, proof of Lemma~\ref{0321-l2}, 
there exists a positive integer $N$ such that for any $s\geq N$, the mapping space $[Y/\Sigma_k; O/O(ks)]$ is trivial.  
By an analogous argument of Step~2, proof of Lemma~\ref{0321-l2}, if 
$s(\xi_Y)\geq N$ is finite, then  $o(\xi_Y)=s(\xi_Y)$. 
\end{remark}

The following lemma is a straightforward observation. 

\begin{lemma}\label{comparison}
Let $i:(M,M_0)\longrightarrow (N,N_0)$ be an injective map of relative finite $CW$-complexes, i.e. the map $i: M\longrightarrow N$ is  injective and it induces a map $i|_{M_0}: M_0\longrightarrow N_0$. Then  \begin{eqnarray}
\label{0229-e1}
 s(\xi_{M|M_0,k})\mid s(\xi_{N|N_0,k}),\text{\ \ \ \ \ }
o(\xi_{M|M_0,k})\mid o(\xi_{N|N_0,k}).
\end{eqnarray}
\end{lemma}
\begin{proof}
The map $i$ induces an injective map between configuration spaces
\begin{eqnarray*}
\tilde i: F(M|M_0,k)\longrightarrow F(N|N_0,k)
\end{eqnarray*}
sending $(x_1,\cdots,x_k)$ to $(i(x_1),\cdots,i(x_k))$. Moreover, $\tilde i$ induces an injective map between unordered configuration spaces
\begin{eqnarray*}
\tilde i/\Sigma_k: F(M|M_0,k)/\Sigma_k\longrightarrow F(N|N_0,k)/\Sigma_k.
\end{eqnarray*}
It is direct to verify that  
\begin{eqnarray}\label{0305-e1}
\xi_{M|M_0,k}\cong (\tilde i/\Sigma_k)^* \xi_{N|N_0,k}.
\end{eqnarray}
Consequently, (\ref{0229-e1}) 
follows from (\ref{0305-e1}). 
\end{proof}

 The following lemma studies the canonical vector bundle over  configuration spaces of the relative $CW$-complex $(M,\{\text{point}\})$.

\begin{lemma}\label{rel-p2}
 Let $M$ be a $CW$-complex with a non-degenerate base-point. Then 
\begin{eqnarray}\label{period-e16}
\xi_{M|\{\text{point}\},k}\cong \xi_{M\setminus \{\text{point}\},k-1}\oplus\epsilon^1. 
\end{eqnarray} 
\end{lemma}
\begin{proof}
It follows from a direct computation that
\begin{eqnarray}\label{rel-e2}
F(M|\{\text{point}\},k)
&\cong& \coprod_k F(M\setminus\{\text{point}\}, k-1),\\
\label{rel-e3}
F(M|\{\text{point}\},k)/\Sigma_k &\cong&   (\coprod_k F(M\setminus\{\text{point}\}, k-1)/\Sigma_{k-1})/\sim \nonumber \\
&\cong&F(M\setminus\{\text{point}\}, k-1)/\Sigma_{k-1} 
 \end{eqnarray}
where $\sim$ is the canonical identification of the $k$ disjoint  components. By (\ref{rel-e2}) and (\ref{rel-e3}) we see that the following commutative diagram gives an isomorphism of vector bundles
\begin{eqnarray*}  
\xymatrix{
\mathbb{R}^k \ar[d]\ar@{=}[r]  &\mathbb{R}^k\ar[d]\\
F(M|\{\text{point}\},k)\times_{\Sigma_k}\mathbb{R}^k\ar[d]\ar[r]  
& F(M\setminus\{\text{point}\},k-1)\times_{\Sigma_{k-1}}\mathbb{R}^{k-1}\times\mathbb{R}\ar[d]\\
F(M|\{\text{point}\},k)/\Sigma_k\ar[r]^{\cong}  
& F(M\setminus\{\text{point}\}, k-1)/\Sigma_{k-1}.
}
\end{eqnarray*}  
Consequently, we have (\ref{period-e16}). 
\end{proof}

The following lemma gives the order as well as the stable order of the canonical vector bundle over configuration spaces of a disjoint union of $CW$-complexes. 

\begin{lemma}\label{disj-l1}
Let $n\geq 2$ and $M_1,\cdots,M_n$ be finite $CW$-complexes. Then the stable order of $\xi_{\coprod_{i=1}^n M_i, k}$ is the smallest common multiple of  
\begin{eqnarray*}
\{s(\xi_{M_i,t})\mid 1\leq i\leq n, 1\leq t\leq k\}. 
\end{eqnarray*} 
And 
 the order of $\xi_{\coprod_{i=1}^n M_i, k}$ is the smallest common multiple of  
\begin{eqnarray*}
\{o(\xi_{M_i,t})\mid 1\leq i\leq n, 1\leq t\leq k\}. 
\end{eqnarray*}
\end{lemma}

\begin{proof}
 It follows from a direct computation that
\begin{eqnarray}\label{disj-e1}
F(\coprod_{i=1}^n M_i, k)/\Sigma_k= \coprod _{\scriptstyle\sum_{i=1}^n k_i=k,   \atop \scriptstyle  k_1,\cdots,k_n\geq 0 } \prod_{j=1}^n F(M_j,k_j)/\Sigma_{k_j}. 
\end{eqnarray}
Here the configuration space of zero point is defined to be the base point. 
Moreover, for any $k_1,\cdots,k_n\geq 0$ such that $\sum_{i=1}^n k_i=k$,  if we denote $\varphi_{k_1,\cdots,k_n}$ as the canonical inclusion of  $\prod_{i=1}^n F(M_i,k_i)/\Sigma_{k_i}$ into $F(\coprod_{i=1}^n M_i,k)/\Sigma_{k}$ given by (\ref{disj-e1}),  then 
\begin{eqnarray}\label{disj-e2}
\prod_{i=1}^n\xi_{M_i,k_i}\cong  \varphi_{k_1,\cdots,k_n}^*\xi_{\coprod_{i=1}^n M_i,k}.
\end{eqnarray}
It follows with the help of (\ref{disj-e2}) that
\begin{eqnarray}\label{0220-e1}
\xi_{\coprod_{i=1}^n M_i,k}&\cong& \coprod_{\scriptstyle\sum_{i=1}^n k_i=k,   \atop \scriptstyle  k_1,\cdots,k_n\geq 0 } \varphi_{k_1,\cdots,k_n}^*\xi_{\coprod_{i=1}^n M_i,k}\nonumber\\
&\cong& \coprod_{\scriptstyle\sum_{i=1}^n k_i=k,   \atop \scriptstyle  k_1,\cdots,k_n\geq 0 }\prod_{i=1}^n\xi_{M_i,k_i}.
\end{eqnarray}
Since the order (resp. the stable order)
 of a product of vector bundles equals to the smallest common multiple of the orders (resp. the stable orders)
  of each factor, and the order  (resp. the stable order)
   of a disjoint union of vector bundles equals to the smallest common multiple of the orders (resp. the stable orders)
    of each component,  Lemma~\ref{disj-l1}  follows from
(\ref{0220-e1}).
\end{proof}

The following lemma characterizes the order of  $\xi_{M,k}$ when $M$ is a connected manifold. 
\begin{lemma}\label{0309-l1}
Let $M$ be a connected $m$-dimensional manifold with $m\geq 2$, and $k\geq 2$. 

(a). Then the first Stiefel-Whitney class $w_1(\xi_{M,k})$ is non-zero.

(b). If there exists an integer $n$, which is a power of $2$, such that for any non-zero element $x$ in $H^1(F(M,k)/\Sigma_k;\mathbb{Z}_2)$, $x^n\neq 0$,  then the order of $\xi_{M,k}$ cannot divide $n$. 
\end{lemma}
\begin{proof}
We first prove (a). We notice that  $F(M,k)$ is connected and the covering map from $F(M,k)$ to $F(M,k)/\Sigma_k$ induces an epimorphism 
\begin{eqnarray*}
h: \pi_1(F(M,k)/\Sigma_k)\longrightarrow \Sigma_k. 
\end{eqnarray*}
Let $r: \Sigma_k\longrightarrow O(k)$ be the regular representation of $\Sigma_k$ given by permuting the coordinates of $\mathbb{R}^k$.   Let $\delta: O(k)\longrightarrow \{\pm 1\}$ be the sign representation of $O(k)$.    Since  $h$ is surjective and 
$\delta\circ r$ is non-trivial,  the map $\delta\circ r\circ h$ is non-trivial.   Moreover, it is direct to verify that  there is a bijection between $\text{Hom}(\pi_1(F(M,k)/\Sigma_k),\mathbb{Z}_2)$  and $\text{Vect}_{\mathbb{R}}^1(F(M,k)/\Sigma_k)$, and this bijection sends $\delta\circ r\circ h$ to the determinant line bundle of $\xi_{M,k}$.   Consequently, the determinant line bundle of $\xi_{M,k}$ is non-trivial.   Therefore, $\xi_{M,k}$ is non-orientable and (a) follows.  

Now we turn to prove (b).  It follows from (a) and the conditions in (b) that
\begin{eqnarray*}
(w_1(\xi_{M,k}))^n\neq 0. 
\end{eqnarray*}
Since $n$ is a power of $2$,
\begin{eqnarray*}
w(\xi_{M,k}^{\oplus n})&=&(1+\sum_{i=1}^{k-1} w_i(\xi_{M,k}))^n\\
&\equiv& 1+ \sum_{i=1}^{k-1} (w_i(\xi_{M,k}))^n \text{ \ \ \ (mod }2)\\ 
&\neq& 0. 
\end{eqnarray*}
Therefore,  $\xi_{M,k}^{\oplus n}$ is not  trivial and (b) follows. 
\end{proof}


\section{The $p$-power of the stable order of the canonical vector bundle over configuration spaces}\label{2.4}

we give some lemmas on the $p$-power of the stable order of the canonical vector bundle over configuration spaces.
\smallskip

For  a finite $CW$-complex $M$ and a prime $p$, we denote $s_p(\xi_{M,k})$ to be the largest $p$-power $p^n$ that can divide   $s(\xi_{M,k})$. We call $s_p(\xi_{M,k})$ the $p$-power of $s(\xi_{M,k})$.   
\begin{lemma}\label{bd-c-p1}
Let $M$ be a finite $CW$-complex.  Then for any prime $p$ and $k\geq p$, 
\begin{eqnarray*}
s_p(\xi_{M,k})\leq s_p(\xi_{M,p}).
\end{eqnarray*}
\end{lemma}
\begin{proof}
The proof  follows from  \cite[p. 105]{bundle1989}. 
\end{proof}

The following lemma is a straight-forward generalization of \cite[Lemma~2.1]{bundle1989}. We give a proof here since the proof of \cite[Lemma~2.1]{bundle1989} is omitted in \cite{bundle1989}. 

\begin{lemma}\label{bd-m-p1}
Let $M$ be a finite $CW$-complex. Suppose $M$ is non-compact. Then for any prime $p$ and $k\geq  p$, 
\begin{eqnarray*}
s_p(\xi_{M,k})=s_p(\xi_{M,p}).
\end{eqnarray*}
\end{lemma}

\begin{proof}
Without loss of generality, we assume $k>p$. 
Since $M$ is a non-compact $CW$-complex, there exist distinct points $a_1,a_2,\cdots,a_{k-p}\in M$ and an embedding $\varphi$ 
from
$M$ into $M\setminus \{a_1,a_2,\cdots,a_{k-p}\}
$.
Hence there is a $\Sigma_p$-equivariant embedding $\Phi$ from $F(M,p)$ into $F(M,k)$ sending $(x_1,$ $\cdots,$ $x_p)$ to $(\varphi(x_1),$ $\cdots,$ $\varphi(x_p),$ $a_1,$ $\cdots,$ $a_{k-p})$.  As a consequence, there is a pull-back diagram of vector bundles
\begin{eqnarray*}
\xymatrix{
\mathbb{R}^p \oplus \mathbb{R}^{k-p}\ar[d]\ar@{=}[rrr] &&&\mathbb{R}^k\ar[d]\\
(F(M,p)\times_{\Sigma_p}\mathbb{R}^p) \oplus \mathbb{R}^{k-p}\ar[rrr]^{\text{\ \ \ \ }(\Phi\times_{\Sigma_p}{\text{Id}_{\mathbb{R}^p}})\times \text{Id}_{\mathbb{R}^{k-p}}} \ar[d]&&&F(M,k)\times_{\Sigma_k}\mathbb{R}^k\ar[d]\\
F(M,p)/\Sigma_p\ar[rrr] ^{ \Phi/\Sigma_p\times (a_1,\cdots,a_{k-p}) 
} &&& F(M,k)/\Sigma_k
}
\end{eqnarray*}
such that 
\begin{eqnarray*}
\xi_{M,p}\oplus \epsilon^{k-p}\cong (\Phi/\Sigma_p\times (a_1,\cdots,a_{k-p}) 
)^*\xi_{M,k}.
\end{eqnarray*}
It follows that 
\begin{eqnarray*}
s_p(\xi_{M,k})\geq s_p(\xi_{M,p}).
\end{eqnarray*}
 With the help of Lemma~\ref{bd-c-p1}, we obtain Lemma~\ref{bd-m-p1}.  
\end{proof}

The following corollary is a consequence of Lemma~\ref{bd-c-p1}  and \cite[Theorem~7.4]{adams}. 
\begin{corollary}\label{01192016-e11} 
The order as well as the stable order of $\xi_{S^m,2}$ is $2^{\rho(m)}$.  Moreover, if $k> 2$, then $s_2(\xi_{S^m,k})\leq 2^{\rho(m)}$.
\end{corollary}

The following corollary is a consequence of  Lemma~\ref{bd-m-p1} and \cite[Theorem~7.4]{adams}. 
\begin{corollary}\label{period-l1}
The order as well as the stable order of $\xi_{\mathbb{R}^m,2}$ is $2^{\rho(m-1)}$.   Moreover, if $k> 2$, then $s_2(\xi_{\mathbb{R}^m,k})= 2^{\rho(m-1)}$.
\end{corollary}


\section{Cohomology of configuration spaces of spheres}\label{2.2}
We give some lemmas on the cohomology of configuration spaces of spheres. We suppose that $M$ is a manifold and $M_0$ is a submanifold of $M$ throughout this section. 
\smallskip

For a  
topological space $X$ with non-degenerate base-point $*$,  
we define the space 
\begin{eqnarray*}
C(M,M_0;X)=\coprod_{k\geq 1}F(M,k)\times _{\Sigma_k} X^k/\approx
\end{eqnarray*}
where $\approx$ is generated by 
\begin{eqnarray*}
(m_1,\cdots,m_k;x_1,\cdots,x_k)\approx (m_1,\cdots,m_{k-1};x_1,\cdots,x_{k-1})
\end{eqnarray*}
 if either $m_k\in M_0$ or $x_k=*$. 
 Such spaces occur as models for mapping spaces (cf.  \cite{mapping1,mapping2}). The space $C(M,M_0;X)$ is filtered by closed subspaces
\begin{eqnarray*}
C_k(M,M_0;X)=\coprod_{j=1}^kF(M,j)\times _{\Sigma_j} X^j/\approx
\end{eqnarray*}
with $C_0(M,M_0;X)$ defined to be the base-point and 
$C_1(M,M_0;X)$ the space $(M/M_0)\wedge X$. 
The inclusions of $C_{k-1}(M,M_0;X)$ into $C_k(M,M_0;X)$ are cofibrations \cite[Theorem~7.1]{mayg}. Their cofibres are denoted by $D_k(M,M_0;X)$, called the $k$-adic construction.  There is  a well-known Snaith splitting (for example, \cite[Proposition~2.4]{wu1})
\begin{eqnarray*}
\Sigma^{\infty}C(M,M_0;X)\simeq \Sigma^{\infty} \bigvee_{k=1}^\infty D_k(M,M_0;X). 
\end{eqnarray*}
Once $M_0$ is the empty set, the spaces $C(M, M_0; X)$, $C_k(M, M_0; X)$ and $D_k(M, M_0; X)$ will be simply denoted as $C(M;  X)$, $C_k(M; X)$ and $D_k(M; X)$ respectively.    

\smallskip

The following lemma gives the rational cohomology of configuration spaces of even-dimensional spheres.

\begin{lemma}\label{rational}\cite{fn,sev, oscar}
Let $d$ be a positive integer and $k\geq 3$. Then 
\begin{eqnarray*}
 H^i(F(S^{2d},k)/\Sigma_k;\mathbb{Q})= \left\{
\begin{aligned}
&\mathbb{Q},&  \text{\ \ \ if } i=0 \text{ or\ } 4d-1, \\
&0,&  \text{\ \ \  otherwise}. 
\end{aligned}
\right.
\end{eqnarray*} 
\end{lemma}

The following lemma is a consequence from \cite[Proposition~17 and Theorem~18]{salvatore}.  
\begin{lemma}\label{0207-l1}
Let $d$ be a positive integer and $p$ be an odd prime. Then
\begin{eqnarray}\label{tor-e1}
\text{Tor}_p( H^k(F(S^{2d},p)/\Sigma_p;\mathbb{Z}))= \left\{
\begin{aligned}
 &\mathbb{Z}_p, &  \text{ if } k=2s(p-1), 1\leq s\leq d-1, \\
&0,&   \text{otherwise. } 
\end{aligned}
\right.
\end{eqnarray}
\end{lemma}

\begin{proof}
 
Let $F$ denote the graded commutative algebra generated by a set of vectors.    Following the notations in \cite{salvatore},
\begin{eqnarray}\label{0204-e1}
H_*(C(S^{2d};S^{0});\mathbb{Z}_p)&\cong&F(\iota, Q)\oplus \iota^{p-1}[\iota,\iota]F(\iota^p,Q)\oplus\nonumber\\
&\text{\  }& \Sigma^{2d}[\iota,\iota]F(\iota,Q)\oplus \Sigma^{2d} F(\iota^p,Q)  
\end{eqnarray}
where $Q$ is the set of all admissible sequences of Dyer-Lashof operations on $\iota$ and $[\iota,\iota]$ except the identity. 
Let $\upsilon$ be the degree that corresponds to the number of particles (cf. \cite[Proposition~17 (3) and Theorem~18~(5)]{salvatore}). Then the homology  $H_*(F(S^{2d},p)/\Sigma_p;\mathbb{Z}_p)$ is isomorphic as a vector space to the subspace of  (\ref{0204-e1}) generated by all monomials of degree $p$.  With the helps of \cite[Proposition~17 (3) and Theorem~18~(5)]{salvatore} and that the Dyer-Lashof operations
\begin{eqnarray*}
Q_i: H_q(-;\mathbb{Z}_p)\longrightarrow H_{pq+i(p-1)}(-;\mathbb{Z}_p)
\end{eqnarray*}
is defined when $i$ and $q$ have the same parity (cf. \cite[p. 537 (b)]{salvatore}), we have
\begin{eqnarray}\label{deg-e1}
\{x\in F(\iota, Q)\mid \upsilon(x)=p\}&=&\mathbb{Z}_p\iota^p\oplus_{ s=1 }^{d-1}(\mathbb{Z}_p Q_{2s} \iota\oplus \mathbb{Z}_p\beta Q_{2s}\iota),\\ 
\label{deg-e2}
\{x\in \iota^{p-1}[\iota,\iota]F(\iota^p, Q)\mid \upsilon(x)=p\}&=&0,\\
\label{deg-e3}
\{x\in \Sigma^{2d}[\iota,\iota]F(\iota, Q)\mid \upsilon(x)=p\}&=&\mathbb{Z}_p\iota^{p-3}\Sigma^{2d}[\iota,\iota],\\
\label{deg-e4}
\{x\in \Sigma^{2d} F(\iota^p, Q)\mid \upsilon(x)=p\}&=&0.
\end{eqnarray}
It follows from  (\ref{0204-e1}) - (\ref{deg-e4}) that
\begin{eqnarray*}
H_*(F(S^{2d},p)/\Sigma_p;\mathbb{Z}_p)&=&\mathbb{Z}_p\iota^p\oplus_{s=1}^{d-1} (\mathbb{Z}_pQ_{2s}\iota\oplus \mathbb{Z}_p \beta Q_{2s}\iota)\\
&&
\oplus\mathbb{Z}_p\iota^{p-3}\Sigma^{2d}[\iota,\iota]. 
\end{eqnarray*} 
Here the dimensions of the generators are 
\begin{eqnarray*}
|\iota^p|&=&0,\\
|\iota^{p-3}\Sigma^{2d}[\iota,\iota]|&=&4d-1,\\
|Q_{2s} \iota|&=& 2s(p-1),\\
 |\beta Q_{2s} \iota|
&=& 2s(p-1)-1. 
\end{eqnarray*} 
Moreover, since $\beta^2=0$, we have
\begin{eqnarray*}
\beta (\iota^{p-3}\Sigma^{2d}[\iota,\iota])=0. 
\end{eqnarray*}
 By the computation of the  Bockstein Spectral Sequence and the Universal Coefficient Theorem for cohomology (cf. \cite[Proof of (3.2)]{swy} or \cite[p. 17]{swy-thesis}), we obtain (\ref{tor-e1}). 
 \end{proof}




\smallskip

\section{Proof of Theorem~\ref{main}}\label{sss4}

The main aim of this section is to prove Theorem~\ref{main}.  
In order to do this, we first study the order  of (\ref{01292016-e1}) when $M$ is a Euclidean space.    

\begin{theorem}\label{period-t1}
Let $k\geq 2$.  Then  the order of $\xi_{\mathbb{R}^m,k}$  is $a_{m,k}$. Moreover, if $k$ is nonprime and $M_0$ is a non-empty $CW$-subcomplex of $\mathbb{R}^m$,   then the order of  $\xi_{\mathbb{R}^m|M_0,k}$ is $a_{m,k}$.
\end{theorem}

\begin{proof}
It follows from \cite[Theorem~1.1]{bundle1983} and Corollary~\ref{period-l1}  that 
$s(\xi_{\mathbb{R}^m,k})=a_{m,k}$. 
And  with the help of Lemma~\ref{0321-l2}, the first assertion follows. 
 
Suppose in addition that $k$ is nonprime and $M_0$ is a non-empty $CW$-subcomplex of $\mathbb{R}^m$. Let $*$ be a point of $M_0$. Then we have the embeddings
\begin{eqnarray}\label{0203-e1}
(\mathbb{R}^m,*)\longrightarrow (\mathbb{R}^m,M_0)\longrightarrow \mathbb(\mathbb{R}^m,\mathbb{R}^m).
\end{eqnarray}
With the help of Lemma~\ref{rel-p2}, we obtain
\begin{eqnarray}
o(\xi_{\mathbb{R}^m|*,k})&=& o(\xi_{\mathbb{R}^m\setminus  \{*\},k-1})\nonumber\\
&=&a_{m,k-1}.\label{0203-e3}  
\end{eqnarray}
Since $k$ is nonprime, $a_{m,k-1}=a_{m,k}$. Therefore, the second assertion follows from the first assertion and (\ref{0203-e3}).  
\end{proof}
\begin{remark}
Theorem~\ref{period-t1} generalizes the conjecture 
$s(\xi_{\mathbb{R}^m,k})=a_{m,k}$
which was made by F.R. Cohen, M.E. Mahowald and S.W. Yang (cf. \cite{bundle1983, swy}). 
\end{remark}

As a consequence of Theorem~\ref{period-t1}, we give the order of (\ref{01292016-e1}) when $M$ is an odd dimensional hypersurface in a Euclidean space. 

\begin{corollary}\label{bd-hypers}
Let $M$ be a hypersurface in $\mathbb{R}^{m+1}$, $m$ odd and $k\geq 2$.  Then
the order  of      $\xi_{M,k}$   is either $a_{m,k}$ or $2^{\rho(m)-\rho(m-1)}a_{m,k}$.  
Moreover, if $k$ is nonprime  
and $M_0$  is a non-empty $CW$-subcomplex of $M$, then  
  the order of $\xi_{M|M_0,k}$ is either $a_{m,k}$ or $2^{\rho(m)-\rho(m-1)}a_{m,k}$.  
 \end{corollary}

\begin{proof}
Suppose $M$ is a hypersurface in $\mathbb{R}^{m+1}$ with $m$ odd. We consider the injective maps 
\begin{eqnarray*}
\mathbb{R}^{m}\overset{i}{\longrightarrow} M\overset{j}{\longrightarrow} \mathbb{R}^{m+1}. 
\end{eqnarray*}
By Lemma~\ref{comparison}, 
\begin{eqnarray}\label{bd-eq1}
o (\xi_{\mathbb{R}^m,k})\mid 
 o (M,k)\mid o (\xi_{\mathbb{R}^{m+1},k}). 
\end{eqnarray}
The first assertion follows from Theorem~\ref{period-t1} and (\ref{bd-eq1}).

Now we suppose in addition that $k$ is nonprime and $M_0$ is a non-empty $CW$-subcomplex of $M$.  Let $*$ be a point of $M_0$. Then by Lemma~\ref{comparison}, the inclusions of relative $CW$-complexes 
\begin{eqnarray*} 
(M,*)\longrightarrow (M,M_0)\longrightarrow (M,M)
\end{eqnarray*}
imply
\begin{eqnarray}\label{rel-e59}
o(\xi_{M|*,k})\mid o(\xi_{M|M_0,k})\mid  o(\xi_{M,k}). 
\end{eqnarray}
By Lemma~\ref{rel-p2} and the first assertion, 
\begin{eqnarray}
o(\xi_{M|*,k})&=&o(\xi_{M\setminus\{*\},k-1})\nonumber\\
&=&a_{m,k-1} \text { or } 2^{\rho(m)-\rho(m-1)}a_{m,k-1},\label{0305-e2}\\
o(\xi_{M,k})&=&a_{m,k} \text { or } 2^{\rho(m)-\rho(m-1)}a_{m,k}. \label{0305-e3}
\end{eqnarray}
Since $k$ is nonprime, we have $a_{m,k}=a_{m,k-1}$. Consequently,  the second assertion follows from  Lemma~\ref{0321-l2}, (\ref{rel-e59}), (\ref{0305-e2}) and (\ref{0305-e3}). 
\end{proof}

Now we give the proof of Theorem~\ref{main}. Theorem~\ref{main}~(a) follows from  Lemma~\ref{disj-l1} and Lemma~\ref{0321-l2}. We  prove Theorem~\ref{main}~(b) and Theorem~\ref{main}~(c) here.

\begin{proof}[Proof of Theorem~\ref{main}~(b)]
When $m$ is odd, Theorem~\ref{main}~(b) follows as a particular case of Corollary~\ref{bd-hypers}. Hence in order to prove Theorem~\ref{main}~(b), we assume $m=2d$ where $d$ is a positive integer.  

Let $p$ be an odd prime.  
Let $K(-)$ denote the abelian group associated with the abelian semi-group of isomorphism classes of complex vector bundles under the Whitney sum operation and $\tilde K(-)$ the reduced generalized cohomology group associated to $K(-)$.  
 
We have an  Atiyah-Hirzebruch Spectral Sequence with $E_2$-page   
\begin{eqnarray}\label{0326-e4}
E_2^{i,j}=H^i(F(S^{2d},p)/\Sigma_p;K^j(*)).
\end{eqnarray} 
This spectral sequence converges to a filtration of $K^{i+j} (F(S^{2d},p)/\Sigma_p)$ in the $E_\infty$-page.
We notice that $K^j(*)$ is isomorphic to  $\mathbb{Z}$ if $j$ is even and  $0$ if $j$ is odd.   Hence with the help of Lemma~\ref{rational},   the only differential whose domain and target are possible to have $\mathbb{Z}$-summands at the same time is
\begin{eqnarray*}
d_{4d-2}: E^{0,2t}_{4d-2}\longrightarrow E^{4d-1,2t-4d+2}_{4d-2}, \text{\ \ }t\in\mathbb{Z}.
\end{eqnarray*}
Since $(4d-1)+(2t-4d+2)$ is odd, $d_{4d-2}$ does not create new torsion parts of $K^0(F(S^{2d},p)/\Sigma_p)$.  Hence all the differentials do not create new  torsion parts of $K^0(F(S^{2d},p)/\Sigma_p)$,  and the $p$-torsion  part of  (\ref{0326-e4}), with $i+j=0$,  converges to  
 \begin{eqnarray*}
\text{Tor}_p(K^0(F(S^{2d},p)/\Sigma_p)). 
\end{eqnarray*}  
It follows with the help of Lemma~\ref{0207-l1} that 
\begin{eqnarray}\label{p3-e2}
|\text{Tor}_p ( \tilde K(F(S^{2d},p)/\Sigma_p))|\leq  \left\{
\begin{aligned}
&p^{d-1},&  \text{\ \ \ if } d\geq 2, \\
&0,&  \text{\ \ \   if }  d=1. 
\end{aligned}
\right.
\end{eqnarray}   
Here $| - |$ denotes the order of a group. 
Since the map of complexification of vector bundles, restricted to the odd torsion part, is injective (cf. \cite[p. 142]{swy}), it follows that
\begin{eqnarray}\label{02152016-e1}
s_p(\xi_{S^{2d},p})&=&s_p(\xi_{S^{2d},p}\otimes\mathbb{C})\nonumber\\
&\leq & \left\{
\begin{aligned}
&|\text{Tor}_p ( \tilde K(F(S^{2d},p)/\Sigma_p))|,& \text{ \ \ \ if } d\geq 2, \\
&1,& \text{ \ \ \ if } d=1. 
\end{aligned}
\right.  
\end{eqnarray}
On the other hand, for any $k\geq p$, it follows from Lemma~\ref{comparison} that 
\begin{eqnarray}\label{02152016-e2}
s_p(\xi_{S^{2d},k})&\geq & s_p(\xi_{\mathbb{R}^{2d},k}) \nonumber\\
&=&p^{d-1}. 
\end{eqnarray}
Consequently, with the help of   Lemma~\ref{bd-c-p1},  it follows that all the inequalities in (\ref{p3-e2}) - (\ref{02152016-e2}) hold and 
\begin{eqnarray}\label{0308-e1}
s_p(\xi_{S^{2d},k})= p^{d-1}. 
\end{eqnarray}
 Moreover, by  Lemma~\ref{comparison},
\begin{eqnarray}\label{0308-e3}
s_2(\xi_{S^{2d},k})&\geq& s_2(\xi_{\mathbb{R}^{2d},k})\nonumber\\
&=&2^{\rho(2d-1)}.
\end{eqnarray}
Therefore,  it follows from  Corollary~\ref{01192016-e11},  (\ref{0308-e1}) and (\ref{0308-e3})  that
\begin{eqnarray*} 
s(\xi_{S^{2d},k})= a_{2d,k}
\text{ \ \ or\ \  } 2^{\rho(2d)-\rho(2d-1)} a_{2d,k}.
\end{eqnarray*}
The first assertion of Theorem~\ref{main}~(b) follows with the help of Lemma~\ref{disj-l1}.  

To prove the second assertion, we suppose that $k$ is nonprime. Then the proof  is essentially the same with the proof of the second assertion of Corollary~\ref{bd-hypers}.  
\end{proof}

\begin{proof}[Proof of Theorem~\ref{main}~(c)]
The first assertion of Theorem~\ref{main}~(c) follows from the following two cases.

{\sc Case~1. } $n\geq 2$. 

Then it follows from Lemma~\ref{disj-l1}  that the stable order of $\xi_{\coprod_n S^m,k}$ is the smallest common multiple of  
\begin{eqnarray*}
\{s(\xi_{S^m,t})\mid 1\leq t\leq k\}. 
\end{eqnarray*}
Hence it follows with the help of Lemma~\ref{bd-c-p1} that for any prime $p\leq k$,  
\begin{eqnarray}\label{disj-e88}
s_p(\xi_{\coprod_n S^m,k})= s_p(\xi_{S^m,p}). 
\end{eqnarray}
Theorem~\ref{main}~(b) gives that for any odd prime $p$,
\begin{eqnarray}\label{disj-e3}
s_p(\xi_{S^m,p})=p^{[\frac{m-1}{2}]}.
\end{eqnarray}
And Corollary~\ref{01192016-e11} gives that 
\begin{eqnarray}\label{disj-e33}
s_2(\xi_{ S^m,2})=2^{\rho(m)}.
\end{eqnarray}
The assertion follows from (\ref{disj-e88}) - (\ref{disj-e33}).

{\sc Case~2. }  $n=1$, $m=2$ and $k$ is even. 

Then it follows from \cite[Theorem 1.11]{birman} or \cite[p. 467]{inv} that 
\begin{eqnarray}
H_1(F(S^2,k)/\Sigma_k; \mathbb{Z})&\cong& \mathbb{Z}_{2k-2}\nonumber\\
&\cong& \mathbb{Z}_2\oplus\mathbb{Z}_{k-1}. \label{0317-e1}
\end{eqnarray}
Let $\beta$ be the Bockstein homomorphism  
 associated with the coefficient sequence
\begin{eqnarray*}
0\longrightarrow\mathbb{Z}_2\longrightarrow\mathbb{Z}_4\overset{}{\longrightarrow}\mathbb{Z}_2\longrightarrow 0. \end{eqnarray*}
By the Universal Coefficient Theorem and (\ref{0317-e1}),  it is direct to verify that $\beta$ is injective. It follows that for any non-zero element $x$ in $H^1(F(S^2,k)/\Sigma_k;\mathbb{Z}_2)$,  
$
x^2=\beta x
$
is non-zero as well.
Consequently, with the help of Lemma~\ref{0309-l1}, we obtain 
\begin{eqnarray}\label{0309-e3}
s(\xi_{S^2,k})>2. 
\end{eqnarray}
It follows from Theorem~\ref{main}~(b) and (\ref{0309-e3}) that
$s(\xi_{S^2,k})=4$.
The assertion follows.

To prove the second assertion, we assume that $k$  is nonprime. Then the proof is essentially the same with the proof of  the second assertion of Theorem~\ref{period-t1}.  
\end{proof}

Finally, we give a proof of Corollary~\ref{2017-c1}. 
\begin{proof}[Proof of Corollary~\ref{2017-c1}]
Given a vector bundle $\xi$ and  positive integers $r$ and $t$, we notice 
that $(\xi^{\oplus r})^{\oplus t}$ is trivial if and only if $\xi^{\oplus rt}$ is trivial.  Hence the order of $\xi^{\oplus r}$ is $o(\xi)/\text{gct}(r,o(\xi))$. By letting $\xi$ be $\xi_{\coprod_nS^m,k}$ and $\xi_{\coprod_nS^m\mid M_0,k}$ respectively, we obtain Corollary~\ref{2017-c1}  from Theorem~\ref{main}~(b).
\end{proof}

\section{Periodicity of suspensions of a cofibre space}\label{2.3}

We review Lemma~\ref{period-l4} on the periodicity of iterated suspensions of certain spaces which is proved in an unpublished manuscript of Professor Frederick R. Cohen.  From Lemma~\ref{period-l4} we derive Corollary~\ref{adic-t1}, which will be used in Section~\ref{sss5}. 

\smallskip

Let $\Pi$ be a subgroup of $\Sigma_k$ and $Z$ a topological space with a free  
$\Pi$-action. Let $X$ be a space with non-degenerate base-point $*$ and $X^{(k)}$ the $k$-fold self-smash product of $X$. Let $D(k,Z,\Pi, X)$ be the cofibre of the natural inclusion from $Z\times _{\Pi}\{*\}$ into  $Z\times_{\Pi} X^{(k)}$.  The following lemma, as well as its proof,  is from an unpublished manuscript given by Professor Frederick R. Cohen. 

\begin{lemma}\label{period-l4}  Suppose the vector bundle 
\begin{eqnarray*} 
\eta_{Z,\Pi,k}: \mathbb{R}^k\longrightarrow Z\times_\Pi\mathbb{R}^k\longrightarrow Z/\Pi
\end{eqnarray*}
 has order $l$. Then for any positive integer $t$, there is a homeomorphism
\begin{eqnarray}\label{period-e6}
D(k,Z,\Pi, \Sigma^{tl}X)\longrightarrow \Sigma^{ktl}D(k,Z,\Pi,X).
\end{eqnarray}
\end{lemma}

\begin{proof} 
Let $\Sigma_k$ act on the $k$-fold Cartesian product $X^k$ from the right by permuting coordinates. This action induces an action of $\Pi$ on $X^{(k)}$.  The trivialization of $\eta_{Z,\Pi,k}^{\oplus tl}$ induces an isomorphism of vector bundles
\begin{eqnarray*}
\xymatrix{
\mathbb{R}^{ktl} \ar[d]\ar[rr]^{\theta} &&\mathbb{R}^{ktl}\ar[d]\\
Z\times_\Pi (\mathbb{R}^{tl})^k\ar[rr]^{\Theta} \ar[d]&& (Z/\Pi)\times (\mathbb{R}^{tl})^k\ar[d]\\
Z/\Pi\ar@{=}[rr] &&Z/\Pi
}
\end{eqnarray*}
for which $\theta$ is a linear isomorphism of $\mathbb{R}^{ktl}$. We notice that the map $\Theta$ extends to a homeomorphism 
\begin{eqnarray}\label{period-e13}
\bar \Theta: Z\times_\Pi (S^{tl})^k\longrightarrow  (Z/\Pi)\times (S^{tl})^k
\end{eqnarray}
by regarding $S^{tl}$ as the one-point compactification of $\mathbb{R}^{tl}$ and sending the added point $\infty$ to itself. Moreover, (\ref{period-e13}) extends to a homeomorphism
\begin{eqnarray}\label{period-e14}
\bar \Theta_X: (Z\times X^{(k)})\times_\Pi (S^{tl})^k\longrightarrow  (Z\times_\Pi X^{(k)})\times (S^{tl})^k
\end{eqnarray}
by sending $X^{(k)}$ to itself via the identity map.

 Let $\Gamma$ denote the subspace of $(Z\times X^{(k)})\times_{\Pi} (S^{tl})^k$ represented by the points $(z,$ $(x_1,$ $\cdots,$ $x_k),$ $(v_1,$ $\cdots,$ $v_k))$ where either   some $v_i$ is $\infty$ in ${S}^{tl}$ or some $x_j$ is $*$ in $X$. Similarly, let $\Lambda$ denote the subspace of $(Z\times_\Pi X^{(k)})\times (S^{tl})^k$ represented by the points $(z,$ $(x_1,$ $\cdots,$ $x_k),$ $(v_1,$ $\cdots,$ $v_k))$ where either some $v_i$ is $\infty$ in ${S}^{tl}$ or some $x_j$ is $*$ in $X$.
We notice that $\bar \Theta_X$ maps $\Gamma$ onto $\Lambda$. Hence $\bar \Theta_X$ induces a homeomorphism on the level of quotient spaces and  gives (\ref{period-e6}).
\end{proof}

The following corollary is a consequence of  Lemma~\ref{comparison} and Lemma~\ref{period-l4}. 

\begin{corollary}\label{adic-t1}
Let $(M,M_0)$ be a relative finite $CW$-complex.  Then there are homotopy equivalences
\begin{eqnarray}\label{2017-6}
\Sigma^{kto(\xi_{M,k})}(D_k(M,M_0;X))\longrightarrow D_k(M,M_0;\Sigma^{to(\xi_{M,k})}X).
\end{eqnarray} 
\end{corollary}

\begin{proof} 
For a $\Sigma_k$-space $Z$, it follows from a direct computation  that
\begin{eqnarray}\label{period-e7}
(Z\vee S^0)\wedge _{\Sigma_k} X^{(k)} 
&\simeq& (Z\times_{\Sigma_k} X^{(k)})/ (Z\times_{\Sigma_k} *)\nonumber\\
&\simeq &D(k,Z,\Sigma_k, X). 
\end{eqnarray}
Since the smash product distributes over the wedge, with the help of (\ref{period-e7}), it follows that  \begin{eqnarray}\label{period-e8}
D_k(M,M_0;X)&\simeq& C_k(M,M_0;X)/C_{k-1}(M,M_0;X)\nonumber\\
&\simeq & (F(M,k)/F(M|M_0,k))\wedge _{\Sigma_k} X^{(k)}\nonumber\\
&\simeq& (F(M,k)\wedge _{\Sigma_k}X^{(k)})/ (F(M|M_0,k)\wedge_{\Sigma_k} X^{(k)})\nonumber\\
&\simeq& ((F(M,k)\vee S^0)\wedge _{\Sigma_k}X^{(k)})/ ((F(M|M_0,k)\vee S^0)\wedge_{\Sigma_k} X^{(k)})\nonumber\\
&\simeq& D(k,F(M,k),\Sigma_k,X)/D(k,F(M|M_0,k),\Sigma_k, X).
\end{eqnarray}
It follows from Lemma~\ref{period-l4} and (\ref{period-e8}) that  there are homotopy equivalences
\begin{eqnarray*}
\Sigma^{kt(o(\xi_{M,k}), o(\xi_{M|M_0,k}))}(D_k(M,M_0;X))\longrightarrow D_k(M,M_0;\Sigma^{t(o(\xi_{M,k}),o(\xi_{M|M_0,k}))}X).
\end{eqnarray*} 
Here $(o(\xi_{M,k}), o(\xi_{M|M_0,k}))$ denotes the smallest common multiple of $ o(\xi_{M,k})$ and $ o(\xi_{M|M_0,k})$. 
By Lemma~\ref{comparison}, we see that  $ o(\xi_{M|M_0,k})$ can always divide $ o(\xi_{M,k})$. Hence the homotopy equivalences (\ref{2017-6}) follows. 
\end{proof}

\section{Stable homotopy types of $k$-adic constructions}\label{sss5}

Let $(M,M_0)$ be a relative $CW$-complex, $t$ be a positive integer  and $k\geq 2$. With the help of the order of (\ref{01292016-e1}), the stable homotopy types of $D_k(M,M_0;\Sigma^nX)$ exhibit a natural  periodic behavior as $n$ varies.

\begin{proposition}\label{adic-c1}
For any space $X$ with a non-degenerate base-point, 
\begin{eqnarray*}
\Sigma^{   kta_{m,k}  }(D_k(\mathbb{R}^m;X))\simeq  D_k(\mathbb{R}^m;\Sigma^{  ta_{m,k}  }X).
\end{eqnarray*}
Moreover,  if $k$ is nonprime and $M_0$ is a non-empty $CW$-subcomplex of $\mathbb{R}^m$,   then 
\begin{eqnarray*}
\Sigma^{   kta_{m,k}  }(D_k(\mathbb{R}^m,M_0;X))\simeq  D_k(\mathbb{R}^m,M_0;\Sigma^{  ta_{m,k}  }X).
\end{eqnarray*}
\end{proposition}

By applying  Theorem~\ref{period-t1} to Corollary~\ref{adic-t1}, we obtain  Proposition~\ref{adic-c1}.    

\begin{proposition}\label{adic-c9}
Let $M$ be a hypersurface in $\mathbb{R}^{m+1}$  and  $m\equiv 3,5,7$ mod $8$.  Then for any space $X$ with a non-degenerate base-point,  
\begin{eqnarray*}
\Sigma^{   kta_{m,k}  }(D_k(M;X))\simeq  D_k(M;\Sigma^{  ta_{m,k}  }X).
\end{eqnarray*}
Moreover, if $k$ is nonprime and $M_0$ is a non-empty $CW$-subcomplex of $M$, then  
 \begin{eqnarray*}\label{01042016-e1}
\Sigma^{   kta_{m,k}  }(D_k(M,M_0;X))\simeq  D_k(M,M_0;\Sigma^{  ta_{m,k}  }X).   
\end{eqnarray*}
\end{proposition}
By applying  Corollary~\ref{bd-hypers} to Corollary~\ref{adic-t1}, we obtain  Proposition~\ref{adic-c9}.

\begin{proposition}\label{0215-c1}
Suppose either    (i). $n=1$, $m=2$ and $k$ is even,  or (ii). $n\geq 2$. Then for any space $X$ with a non-degenerate base-point, 
 \begin{eqnarray*} 
\Sigma^{  2^{\rho(m)-\rho(m-1)} kta_{m,k}  }(D_k(\coprod_n S^{m};X))\simeq  D_k(\coprod_n S^{m};\Sigma^{ 2^{\rho(m)-\rho(m-1)} ta_{m,k}  }X).   
\end{eqnarray*}
Moreover, if $k$ is nonprime and $M_0$  is  a non-empty $CW$-subcomplex of $\coprod_n S^{m}$, then  
 \begin{eqnarray*}\label{03262016-e1}
\Sigma^{2^{\rho(m)-\rho(m-1)} kta_{m,k}  }(D_k(\coprod_n S^{m},M_0;X))\simeq  D_k(\coprod_n S^{m},M_0;\Sigma^{ 2^{\rho(m)-\rho(m-1)}  ta_{m,k}  }X).   
\end{eqnarray*}
\end{proposition}
By applying Theorem~\ref{main}~(c)  to Corollary~\ref{adic-t1}, we obtain  Proposition~\ref{0215-c1}. 

\section{Further discussions}

We briefly address and discuss further questions. 
Suppose $\Sigma_k$ acts on $\mathbb{R}^L$ freely. Let $Y$ be a nonempty $\Sigma_k$-invariant subspace of $\mathbb{R}^L$.  Then 
we have an induced free $\Sigma_k$-action on $Y$ 
and an associated vector bundle (cf. Remark~\ref{remark2.2})
\begin{eqnarray*}
\xi_Y: \mathbb{R}^k\longrightarrow Y\times_{\Sigma_k}\mathbb{R}^k\longrightarrow Y/\Sigma_k. 
\end{eqnarray*}
In particular, $\xi_{\mathbb{R}^L}$ is the associated vector bundle of the covering map $\pi: \mathbb{R}^L\longrightarrow\mathbb{R}^L/\Sigma_k$.  The embedding $\iota: Y\longrightarrow \mathbb{R}^L$ gives an induced embedding $\iota/\Sigma_k: Y/\Sigma_k\longrightarrow \mathbb{R}^L/\Sigma_k$.  And $\xi_Y$ is the pull-back vector bundle $(\iota/\Sigma_k)^*\xi_{\mathbb{R}^L}$. 
 Therefore,  by an analogous argument of Lemma~\ref{comparison}, if  $o(\xi_{\mathbb{R}^L})$ (resp. $s(\xi_{\mathbb{R}^L})$) is finite, then $o(\xi_Y)$ (resp. $s(\xi_Y)$) is finite as well  and $o(\xi_{Y})\mid o(\xi_{\mathbb{R}^L})$ (resp. $s(\xi_{Y})\mid s(\xi_{\mathbb{R}^L})$).   

\begin{enumerate}[$\text{\ }$]
\item
{\bf Question~1}.  What features of the $\Sigma_k$-action on $\mathbb{R}^L$  could ensure  that  $\xi_{\mathbb{R}^L}$ is of finite (stable) order?
\item
{\bf Question~2}. What features of $Y$   determine the (stable) order of $\xi_Y$?
\end{enumerate}



\bigskip

\bigskip

Addresses:

$^{\it a}$ School of Mathematics and Computer Science, Guangdong Ocean University, 1 Haida Road, Zhanjiang, China, 524088.
 
 $^{\it b}$ Department of Mathematics,   National University of Singapore, Singapore, 119076.

\smallskip

Email Address: sren@u.nus.edu 



\end{document}